\documentclass[11pt]{article}

\usepackage{mathpazo}   % With old-style figures and real smallcaps.
\usepackage{tikz}

\usepackage{nameref}
\usepackage{empheq}
\usepackage{comment}
\usepackage[shortlabels,inline]{enumitem}
\setlist[enumerate]{nosep}
\usepackage[colorlinks=true,
linkcolor=refkey,
urlcolor=lblue,
citecolor=red]{hyperref}
\usepackage[doc,wmm,hhb]{optional}
\usepackage{xcolor}

\usepackage{float}
\usepackage{soul}
\usepackage{graphicx}
\definecolor{labelkey}{rgb}{0,0.08,0.45}
\definecolor{refkey}{rgb}{0,0.6,0.0}
\definecolor{Brown}{rgb}{0.45,0.0,0.05}
\definecolor{lime}{rgb}{0.00,0.8,0.0}
\definecolor{lblue}{rgb}{0.5,0.5,0.99}
\definecolor{OliveGreen}{rgb}{0,0.6,0}
\definecolor{tyrianpurple}{rgb}{0.4, 0.01, 0.24}
\colorlet{hlcyan}{cyan!30}

\usepackage{stmaryrd}
\usepackage{amssymb}

\makeatletter
\def\namedlabel#1#2{\begingroup
   \def\@currentlabel{#2}%
   \label{#1}\endgroup
}
\makeatother

\newcommand{\seppthree}{\setlength{\itemsep}{-3pt}}

\usepackage[margin=0.92in,footskip=0.25in]{geometry}
\newcommand{\J}[1]{\ensuremath{{\operatorname{J}}_{#1}}}
\newcommand{\R}[1]{\ensuremath{{\operatorname{R}}_%
{#1}}}
\newcommand{\Pj}[1]{\ensuremath{{\operatorname{P}}_%
{#1}}}

\providecommand{\siff}{\Leftrightarrow}

\newcommand{\nnn}{\ensuremath{{n\in{\mathbb N}}}}

\newcommand{\menge}[2]{\big\{{#1}~\big |~{#2}\big\}}

\newcommand{\fenv}[1]%
{\ensuremath{\,\overrightarrow{\operatorname{env}}_{#1}}}
\newcommand{\benv}[1]%
{\ensuremath{\,\overleftarrow{\operatorname{env}}_{#1}}}

\newcommand{\scal}[2]{\left\langle{#1},{#2}  \right\rangle}

\newcommand{\RR}{\ensuremath{\mathbb R}}

\newcommand{\dom}{\ensuremath{\operatorname{dom}}\,}

\newcommand{\prox}{\ensuremath{\operatorname{P}}}

\newcommand{\T}{\ensuremath{{\operatorname{T}}}}
\newcommand{\ran}{\ensuremath{{\operatorname{ran}}\,}}
\newcommand{\zer}{\ensuremath{\operatorname{zer}}}

\newcommand{\cdom}{\ensuremath{\overline{\operatorname{dom}}\,}}

\newcommand{\Id}{\ensuremath{\operatorname{Id}}}

\newcommand{\TAB}{\operatorname{T}_{A,B}}

\newcommand{\sperp}{{\scriptscriptstyle\perp}}%for SMALLER math symbol

{\begin{list}{}{%
\settowidth{\labelwidth}{\textrm{#1~}}%
\setlength{\leftmargin}{\labelwidth+\labelsep}}}%requires macro calc.sty
{\end{list}}
\usepackage{amsthm}
\makeatletter% This is to fix the bold after the theorem
\def\th@plain{%
	\thm@notefont{}% same as heading font
	\itshape % body font
}
\def\th@definition{%
	\thm@notefont{}% same as heading font
	\normalfont % body font
}
\makeatother
\usepackage[capitalize,nameinlink]{cleveref}
\crefname{equation}{}{equations}
\crefname{chapter}{Appendix}{chapters}
\crefname{item}{}{items}
\crefname{enumi}{}{}
\newtheorem{theorem}{Theorem}[section]
\newtheorem{lemma}[theorem]{Lemma}

\newtheorem{corollary}[theorem]{Corollary}

\newtheorem{proposition}[theorem]{Proposition}

\newtheorem{example}[theorem]{Example}
\newtheorem{ex}[theorem]{Example}
\newtheorem{fact}[theorem]{Fact}
\newtheorem{remark}[theorem]{Remark}

\providecommand{\ds}{\displaystyle}

\providecommand{\norm}[1]{\lVert#1\rVert}

\providecommand{\stb}[1]{\left\{#1\right\}}

\providecommand{\LA}{\Leftarrow}
\providecommand{\RA}{\Rightarrow}

\providecommand{\RR}{\mathbb{R}}

\providecommand{\ran}{\operatorname{ran}}

\providecommand{\intr}{\operatorname{int}}

\providecommand{\dom}{\operatorname{dom}}

\newcommand{\fix}{\ensuremath{\operatorname{Fix}}}

\providecommand{\gr}{\operatorname{gra}}
\providecommand{\gra}{\operatorname{gra}}
\providecommand{\Id}{\operatorname{{ Id}}}

\providecommand{\fady}{\varnothing}

\providecommand{\rras}{\rightrightarrows}

\providecommand{\ball}[2]{\operatorname{ball}(#1;#2)}

\providecommand{\gr}{\operatorname{gra}}
\providecommand{\fix}{\operatorname{Fix}}
\providecommand{\ran}{\operatorname{ran}}
\providecommand{\rec}{\operatorname{rec}}
\providecommand{\Id}{\operatorname{Id}}

\providecommand{\zer}{\operatorname{zer}}
\providecommand{\R}{{ R}}
\providecommand{\T}{{ T}}

\providecommand{\fady}{\varnothing}

\newcommand{\cran}{\ensuremath{\overline{\operatorname{ran}}\,}}

\providecommand{\ri}{\operatorname{ri}}

\providecommand{\RR}{\mathbb{R}}

\definecolor{myblue}{rgb}{0.9,0.9,0.98}

  \newcommand*\mybluebox[1]{%
    \colorbox{myblue}{\hspace{1em}#1\hspace{1em}}}

\allowdisplaybreaks % or locally if problems {\allowdisplaybreaks
\begin{document}

\setlength{\abovedisplayskip}{8pt}
\setlength{\belowdisplayskip}{8pt}	
	
%-------------------------------------------------------------------------

%\tikzstyle{decision} = [diamond, draw, fill=blue!50]
%\tikzstyle{line} = [draw, -stealth, thick]
%\tikzstyle{elli}=[draw, ellipse, fill=red!50,minimum height=8mm, text width=5em, text centered]
%\tikzstyle{block} = [draw, rectangle, fill=blue!50, text width=8em, text centered, minimum height=15mm, node distance=10em]
%

\author{
%Heinz H.\ Bauschke\thanks{
%Mathematics, University
%of British Columbia,
%Kelowna, B.C.\ V1V~1V7, Canada. E-mail:
%\texttt{heinz.bauschke@ubc.ca}.}~~~and~
Walaa M.\ Moursi\thanks{
Department of Combinatorics and Optimization, 
University of Waterloo,
Waterloo, Ontario N2L~3G1, Canada.
  %and
  %Mansoura University, Faculty of Science,
  %Mathematics Department,
  %Mansoura 35516, Egypt.
  E-mail: \texttt{walaa.moursi@uwaterloo.ca}.}
}

\title{\textsf{
The range of the Douglas--Rachford operator\\ in 
infinite-dimensional Hilbert spaces
}
}

\date{June 14, 2022}

\maketitle

\begin{abstract}
The Douglas--Rachford algorithm is one of the 
most prominent splitting algorithms for solving 
convex optimization problems.
Recently, the method has been successful in finding
a \emph{generalized solution} (provided that one exists) for optimization problems 
in the inconsistent case, i.e., 
when a solution does not exist.
The convergence analysis of the inconsistent case hinges on the study 
 of the range of the displacement operator associated with the 
Douglas--Rachford splitting operator and the corresponding minimal displacement vector.
In this paper, we provide a formula for the range of
the Douglas--Rachford splitting operator in (possibly) infinite-dimensional Hilbert spaces
under mild assumptions on the underlying operators.
Our new results  complement
known results in finite-dimensional Hilbert spaces.
Several examples illustrate and tighten our conclusions.
\end{abstract}
{ 
\small
\noindent
{\bfseries 2010 Mathematics Subject Classification:}
{Primary 
47H05, %Monotone operators and generalizations
47H09,
90C25;
Secondary 
90C46,
47H14, 
49M27, 
49M29, 
49N15.% General operator theory
}

\noindent {\bfseries Keywords:}
convex function,
convex set,
displacement map,
Douglas--Rachford splitting operator,
duality
maximally monotone operator,
range,
subdifferential operator
}

\section{Introduction}

Throughout, we assume that 
\begin{empheq}[box=\mybluebox]{equation}
\text{$X$ is
a real Hilbert space with inner product 
$\scal{\cdot}{\cdot}\colon X\times X\to\RR$, }
\end{empheq}
and induced norm $\|\cdot\|$.
Let $A\colon X\rras X$.
Recall that $A$ is \emph{monotone}
if $\{(x,u),(y,v)\}\subseteq \gra A$
implies that $\scal{x-y}{u-v}\ge 0$
and that 
$A$ is \emph{maximally monotone}
if any proper extension of 
$\gra A$ destroys its monotonicity.
The \emph{resolvent} of $A$ is 
$\J{A}=(\Id+A)^{-1}$
and the  \emph{reflected resolvent} of $A$ 
is $\R{A}=2\J{A}-\Id$,
where
$\Id\colon X\to X\colon x\mapsto x$.
%\section{Background and auxiliary results}
We recall the well-known 
inverse resolvent identity (see \cite[Lemma~12.14]{Rock98})
\begin{equation}%\label{inv:res}
	\label{e:iri}
	\J{A} +\J{A^{-1}}=\Id,
\end{equation}
and the following, useful description of the graph of $A$
thanks to Minty (see  \cite{Minty}).

\begin{fact}[{Minty}]
	\label{thm:minty}
	Let $A\colon X\rras X$ be monotone. Then
	\begin{equation}
		\label{eq:Minty}
		\gra A=\menge{(\J{A}  x, \J{A^{-1}}x)}{x\in \ran (\Id+A)}.
	\end{equation}
	Moreover, $A$ is maximally monotone $\siff$ $\ran (\Id+A)=X$.
%	\begin{equation}
%		\label{eq:Minty:2}
%		\text{$A$ is maximally monotone $\siff$ $\ran (\Id+A)=X$.}
%	\end{equation}
\end{fact}	

In the following, we assume that 
\begin{empheq}[box=\mybluebox]{equation}
	\label{e:def:T}
	\text{$A$ and $B$ are maximally monotone on $X$.}
\end{empheq}
The Douglas--Rachford splitting
operator associated with the ordered pair $(A,B)$
is 
\begin{empheq}[box=\mybluebox]{equation}
	\label{e:T:sp:def}
	T =  \TAB := \Id -\J{A} + \J{B}\R{A}. 
\end{empheq}
It is critical to observe that  $\zer(A+B)\neq \fady$ 
if and only if $\fix T\neq \fady$, see, e.g.,
\cite[Proposition~26.1(iii)(b)]{BC2017}.
Let $x_0\in X$.
The splitting operator in \cref{e:T:sp:def} 
defines the so-called \emph{governing} sequence
via  iterating $T $ at $x_0$ to obtain
$(x_n)_\nnn= (T^nx_0)_\nnn$
and $(x_n)_\nnn$ converges weakly to 
a point in $\fix T$ provided that the latter is nonempty.
In this case,
the \emph{shadow} sequence 
$(y_n)_\nnn= (\J{A}x_n)_\nnn$
converges weakly to a point in $\zer(A+B)$.
Note that 
\begin{equation}
	\label{e:Id-T:desc}
\Id-T = \J{A}-\J{B}\R{A}=\J{A^{-1}}+\J{B^{-1}}\R{A}.
\end{equation}
The range of the displacement operator $\Id-T$
is the central focus of this paper. 
Indeed, $\fix T\neq \fady$
if and only if $0\in \ran (\Id-T)$.
%It natural to ask whether and how this range is 
%related to the operators $A$ and $B$. 
Because $T$ is firmly nonexpansive, we learn that
(see, e.g., 
 \cite[Theorem~31.2]{Simons2})
 %\cite[Example~20.19~and~Corollary~21.14]{BC2017}
\begin{equation}
\label{e:prop:disp}	
\text{ $\Id-T$ is maximally monotone, hence $\cran(\Id-T)$ is convex.}
\end{equation}
This beautiful topological property 
of $\cran (\Id-T)$
allows us to work with best approximation properties of
nonempty closed convex sets.
In the following we set
\begin{equation}
D=\dom A-\dom B
\text{\;\;and\;\;} R=\ran A +\ran B.	
\end{equation}	
It is always true that  
(see \cite[Corollary~4.1]{EckThesis} or \cite[Corollary~2.14]{Sicon})
\begin{equation}
\ran(\Id-T) = \menge{a-b}{(a,a^*)\in\gr A,\, (b,b^*)\in \gr B,\, a-b=a^*+b^*} 
\subseteq D \cap R.
\end{equation}
Hence 
\begin{equation}
	\label{eq:inc:always}
\cran(\Id-T)\subseteq \overline{D\cap R}.
%\subseteq \overline{D}\cap\overline{R}
%.
\end{equation} 
When $X$ is finite-dimensional, 
the authors in \cite{MOR} proved that under mild assumptions on
$A$ and $B$ we have 
\begin{empheq}[box=\mybluebox]{equation}
	\label{e:assump:free:fd}
	\cran(\Id-T)=\overline{D\cap R}.
%	=\overline{D}\cap \overline{R}. 
\end{empheq}

\emph{Contribution.} In this paper we prove that
the identity in \cref{e:assump:free:fd} 
holds in  infinite-dimensional Hilbert spaces.
Our proof techniques are  completely independent 
and significantly distinct
from the 
tools used in the finite-dimensional case:
%to prove  in \cite{MOR}:
Indeed the analysis techniques used in \cite{MOR}
hinge on the well-developed  calculus of the relative interior of 
convex and \emph{nearly convex}\footnote{Let $E$ be a subset of $X$.
	We say that $E $ is \emph{nearly convex}
	%(see \cite[Theorem~12.41]{Rock98})
	if there exists a convex subset $C$ of
	$X$ such that $C\subseteq E\subseteq \overline{C}$.
	(For 
	detailed discussion on the algebra of nearly
	convex sets we refer the reader to \cite[Section~3]{Rock70}.)} sets (see \cite{Rock1970}
 and \cite{BMW}). Not only does this powerful calculus 
 notably fails to extend in infinite-dimensional Hilbert spaces
 but also there exists no analogously useful notion 
 that can replace it. This highlights the 
 infeasibility of extending any of the proof techniques in \cite{MOR}
to infinite-dimensional settings.
 Instead, our analysis hinges on a novel and 
 powerful approach that connects 
 the graph of $\Id-T$
 and the graphs of the individual operators $A$
 and $B$. This allows to exploit any additional properties 
 of the operators $A$
  and $B $ and demonstrate how this reflects on 
  the displacement mapping $\Id-T$.
  As a byproduct of our analysis we were able to 
refine the topological properties of the sets $D$ and $R$
as well as the corresponding minimal norm vectors.

\emph{Organization and notation.}
The remainder of this paper is organized as follows:
\cref{sec:2} presents the main results of the paper.
In \cref{sec:3} we provide sharper conclusions 
when specializing the main results to subdifferential operators.
Finally, in \cref{sec:4} we provide a detailed study of the 
minimal norm vector in the range of displacement map.
An application of our results to the product space setting is also 
presented.

 Our notation is standard and follows largely, e.g., 
\cite{BC2017}, \cite{Rock98} and \cite{Simons1}.

\section{The range of $\Id-T$}
\label{sec:2}
We start with the following key lemma. 
\begin{lemma}
	\label{lem:gr:Id-T}
	Let $(a,b,a^*,b^*,x)\in X^5$.
	Then the following holds:	
\begin{equation}
\begin{split}
		\label{e:subgrad:scal}
\scal{x-(a+a^*)}{x-Tx -(a-b)}
&\ge 
\scal{\J{A}x-a}{\J{A^{-1}}x-a^*}
\\
&\quad+\scal{\J{B}\R{A} x-b}{\J{B^{-1}}\R{A}x-(a-a^*-b)}.
\end{split}
	\end{equation}
	\end{lemma}	
\begin{proof}
Recalling  \cref{e:Id-T:desc}, we have 	
\begin{subequations}
\begin{align}
&\qquad\scal{x-(a+a^*)}{x-Tx -(a-b)}	-\scal{\J{A}x-a}{\J{A^{-1}}x-a^*}
\nonumber
\\
&\qquad-\scal{\J{B}\R{A} x-b}{\J{B^{-1}}\R{A}x-(a-a^*-b)}
\nonumber
\\
&=\scal{x-(a+a^*)}{\J{A}x -a-(\J{B}\R{A} x-b)}
-\scal{\J{A}x-a}{x-(a+a^*)-(\J{A}  x -a)}
\nonumber
\\
&\quad+\scal{\J{B}\R{A}x-b}{2(\J{A}x-a)-\J{B}\R{A} x-(x-a-a^*-b)}
\\
&=\scal{x-(a+a^*)}{\J{A}x -a}
-\scal{x-(a+a^*)}{\J{B}\R{A} x-b}
\nonumber
\\
&
\quad-\scal{\J{A}x-a}{x-(a+a^*)}+\norm{\J{A} x -a}^2-2\scal{\J{A}x-a}{\J{B}\R{A} x-b}
\nonumber
\\
&\quad+\norm{\J{B}\R{A}x-b}^2+\scal{\J{B}\R{A}x-b}{x-(a+a^*)}
\\
&=\norm{\J{A} x -a}^2-2\scal{\J{A}x-a}{\J{B}\R{A} x-b}+\norm{\J{B}\R{A}x-b}^2
\\
&=\norm{x-Tx-(a-b)}^2\ge 0.
\end{align}
\end{subequations}
The proof is complete.	
\end{proof}	
\begin{proposition}
	\label{prop:lin:case:gen}
Suppose that 
%$(\forall C\in \{A,B\})$ we have 
%$	(\dom C-\dom C)\perp (\ran C-\ran C).$
\begin{equation}
	\label{eq:ortho}
	(\dom A-\dom A)\perp (\ran A-\ran A)
	\text{~~and ~~}
	(\dom B-\dom B)\perp (\ran B-\ran B).
\end{equation}		
Then $\ran (\Id-T)=({\dom}\ A-{\dom}\ B)\cap ({\ran}\ A+{\ran}\ B)$.
\end{proposition}	

\begin{proof}
Indeed, let 
$w\in ({\dom}\ A-{\dom}\ B)\cap ({\ran}\ A+{\ran}\ B)$.
Then $(\exists (a,b,a^*,b^*) \in 
{\dom} A\times{\dom} B\times {\ran} A\times{\ran} B$
such that $	w=a-b=a^*+b^*$.
\cref{lem:gr:Id-T}  implies that $(\forall x\in X)$
\begin{align}
\label{e:gra:Id-T:magic}	
\scal{x-(a+a^*)}{x-Tx -w}
	&\ge 
	\scal{\J{A}x-a}{\J{A^{-1}}x-a^*} 
	+\scal{\J{B}\R{A} x-b}{\J{B^{-1}}\R{A}x-b^*}.
\end{align}
Recalling \cref{thm:minty}, we learn
that  
\begin{subequations}
	\label{e:ortho:AB}
\begin{align}
({\J{A}x-a},{\J{A^{-1}}x-a^*} )&\in (\dom A-\dom A)\times (\ran A-\ran A)
\\
({\J{B}\R{A} x-b},{\J{B^{-1}}\R{A}x-b^*})&\in (\dom B-\dom B)\times (\ran B-\ran B).
\end{align}
\end{subequations}	
Combining \cref{e:gra:Id-T:magic} and \cref{e:ortho:AB}
in view of \cref{eq:ortho}
yields $	\scal{x-(a+a^*)}{x-Tx -w}
\ge 0.$
In view of \cref{e:prop:disp}	
 we conclude that 
$(a+a^*,w)\in \gra (\Id-T)$. This completes the proof. 
\end{proof}

\begin{example}
	\label{ex:line:subsp}
Let $U$ and $V$ be closed linear subspaces of $X$,
let $(u,v)\in U\times V$ and let $(u^\sperp,v^\sperp)\in U^\perp\times V^\perp$.
Set $(A,B)=(u+N_{u^\sperp+U},v+N_{v^\sperp+V})$.
Then the following hold:
\begin{enumerate}
\item
\label{ex:line:subsp:i}
$(\dom A,\dom B,\ran A,\ran B)=(u^\sperp+U,v^\sperp+V,u+U^\perp,v+V^\perp)$.
\item
\label{ex:line:subsp:ii}
$(\dom A-\dom A,\dom B-\dom B,\ran A-\ran A,\ran B-\ran B)=(U,V,U^\perp,V^\perp)$.
\item	
\label{ex:line:subsp:iii}
$(\forall C\in \{A,B\})$ we have 
$	(\dom C-\dom C)\perp (\ran C-\ran C)$.
\item	
\label{ex:line:subsp:iv}
$\ran (\Id-T)=(u^\sperp-v^\sperp+U+V)\cap(u+v+U^\perp+V^\perp)$.
\end{enumerate}	
\end{example}	
\begin{proof}
\cref{ex:line:subsp:i}--\cref{ex:line:subsp:ii}:
This is clear.
\cref{ex:line:subsp:iii}:
This is a direct consequence of \cref{ex:line:subsp:ii}.

\cref{ex:line:subsp:iv}:
Combine \cref{ex:line:subsp:iii} and 
\cref{prop:lin:case:gen}.	
\end{proof}	

We are now ready to derive an analogous formula to 
the  conclusion of \cref {prop:lin:case:gen}
in a more general setting.
Let $C\colon X\rras X$ be monotone. 
Recall that
$C$ is
\emph{$3^*$ monotone} (this is also known as \emph{rectangular}) 
if $(\forall (y,z^*)\in \dom C\times \ran C)$ 
we have $
\inf_{(x,x^*)\in\gra C}\scal{x-y}{x^*-z^*}>-\infty$.
It is well-known that
(see, e.g.,
\cite[page~167]{Br-H} and 
 \cite[page~127]{Simons2})
for a proper lower semicontinuous convex function $f\colon X\to \left]-\infty,+\infty\right]$
we have 
\begin{equation}
	\text{$\partial f$ is $3^*$ monotone.}	
\end{equation}

\begin{proposition}
\label{prop:opp:inclusion}
Let $w\in ({\dom}\ A-{\dom}\ B)\cap ({\ran}\ A+{\ran}\ B)$.
Suppose that one of the following hold:
\begin{enumerate}
\item
\label{prop:opp:inclusion:i}
$A$ and $B$ are $3^*$ monotone.
\item
\label{prop:opp:inclusion:ii}
$\dom A\subseteq (w+\dom B)$  and $B$ is $3^*$ monotone.
\item
\label{prop:opp:inclusion:iii}
$\dom B\subseteq (-w+\dom A)$  and $A$ is $3^*$ monotone.
\end{enumerate}	
Then $w\in \overline{\ran} (\Id-T)$.
\end{proposition}	
\begin{proof}
Let  $n\ge 1$
and observe that \cref{thm:minty} applied with $A$
replaced by $\Id-T$,
in view of 
\cref{e:prop:disp}
implies that $\ran \Big(\big(1+\tfrac{1}{n^2}\big)\Id-T\Big)=X$.
Consequently, $(\forall n\ge 1)$ $(\exists x_n\in X)$
such that $w=\big(1+\tfrac{1}{n^2}\big)x_n-Tx_n$.
Let 
$(x_n)_{n\ge 1}$ be such that 
$\big(x_n, w-\tfrac{1}{n^2}x_n\big)_{n\ge 1}$ lies in $\gra (\Id-T)$.
It is sufficient to show that 
\begin{equation}
\label{e:main:limit}
\tfrac{1}{n^2}x_n\to 0.
\end{equation}
To this end,  let $x\in X$.
First suppose that 		
\cref{prop:opp:inclusion:i} holds.
By assumption $(\exists (a,b,a^*,b^*) \in 
{\dom} A\times{\dom} B\times {\ran} A\times{\ran} B)$
such that 
\begin{equation}
	\label{e:w:connect}
	w=a-b=a^*+b^*.
\end{equation}	
Now \cref{lem:gr:Id-T} and \cref{e:w:connect} imply
\begin{align}
	\label{e:sq:version}
	\scal{x-(a+a^*)}{x-Tx -w}
	&\ge 
	\scal{\J{A}x-a}{\J{A^{-1}}x-a^*} 
	+\scal{\J{B}\R{A} x-b}{\J{B^{-1}}\R{A}x-b^*}.
\end{align}
It follows from the $3^*$ monotonicity of
$A $ and $B$ that 
\begin{equation}
\label{eq:220519:1}
\inf_{x\in X}\scal{\J{A}x-a}{\J{A^{-1}}x-a^*}  >-\infty \text{~~and~~} \inf_{x\in X}\scal{\J{B}\R{A} x-b}{\J{B^{-1}}\R{A}x-b^*}>-\infty.
\end{equation}	
Combining \cref{eq:220519:1} and  \cref{e:sq:version}, we learn that
\begin{subequations}
	\label{e:sq:version:inf}
\begin{align}
	&\qquad \inf_{x\in X}\scal{x-(a+a^*)}{x-Tx -(a-b)}
	\nonumber
	\\
	&\ge 
	\inf_{x\in X}(\scal{\J{A}x-a}{\J{A^{-1}}x-a^*}  +\scal{\J{B}\R{A} x-b}{\J{B^{-1}}\R{A}x-b^*})
	\\
	&\ge 
		\inf_{x\in X}\scal{\J{A}x-a}{\J{A^{-1}}x-a^*}  +\inf_{x\in X}\scal{\J{B}\R{A} x-b}{\J{B^{-1}}\R{A}x-b^*}>-\infty.
\end{align}
\end{subequations}
It follows from \cref{e:sq:version:inf} that $(\exists M\in \RR)$
such that $(\forall n\ge 1)$ we have $\scal{x_n-(a+a^*)}{-\tfrac{1}{n^2}x_n}\ge M$.
 Using Cauchy--Schwarz, we have
$\tfrac{1}{n^2}\norm{x_n}^2\le \scal{\tfrac{1}{n^2}x_n}{a+a^*}-M\le \tfrac{1}{n^2}\norm{x_n}\norm{a+a^*}-M\le \tfrac{1}{2n^2}(\norm{x_n}^2+\norm{a+a^*}^2)-M$.
Simplifying yields  $\tfrac{1}{n^2}\norm{x_n}^2\le\tfrac{1}{n^2}\norm{a+a^*}^2-2M\le \norm{a+a^*}^2-2M$
and we learn that the sequence $(\tfrac{x_n}{n})_{n\ge 1}$ is bounded.
Consequently, $\tfrac{1}{n^2}x_n\to 0$ and \cref{e:main:limit} is verified.

Now suppose that 		
\cref{prop:opp:inclusion:ii} holds.
By assumption $(\exists (a^*,b^*) \in  {\ran} A\times{\ran} B)$
such that 
\begin{equation}
	\label{e:w:connect:ii}
	w=a^*+b^*.
\end{equation}	
Let $(a,b)\in \dom A\times \dom B$ be such that $(a,a^*) \in \gra A$
 and $a-b=w$.
 \cref{lem:gr:Id-T} and \cref{e:w:connect:ii} imply
 \begin{equation}
 	\label{e:sq:version:ii}
 	\scal{x-(a+a^*)}{x-Tx -w}
 	\ge 
 	\scal{\J{A}x-a}{\J{A^{-1}}x-a^*} 
 	+\scal{\J{B}\R{A} x-b}{\J{B^{-1}}\R{A}x-b^*}.
 \end{equation}
It follows from the monotonicity of 
$A $ and the $3^*$ monotonicity of $B$  that 
\begin{equation}
\scal{\J{A}x-a}{\J{A^{-1}}x-a^*}   \ge 0 \text{~~and~~} \inf_{x\in X}\scal{\J{B}\R{A} x-b}{\J{B^{-1}}\R{A}x-b^*}>-\infty.
\end{equation}	
Combining this with \cref{e:sq:version:ii} we learn that
\begin{subequations}
	\label{e:sq:version:inf:ii}
	\begin{align}
		&\qquad \inf_{x\in X}\scal{x-(a+a^*)}{x-Tx -(a-b)}
		\nonumber
		\\
		&\ge 
		\inf_{x\in X}\big(\scal{\J{A}x-a}{\J{A^{-1}}x-a^*}  +\scal{\J{B}\R{A} x-b}{\J{B^{-1}}\R{A}x-b^*}\big)
		\\
		&\ge 
		\inf_{x\in X}\scal{\J{A}x-a}{\J{A^{-1}}x-a^*} +\inf_{x\in X}\scal{\J{B}\R{A} x-b}{\J{B^{-1}}\R{A}x-b^*}>-\infty.
	\end{align}
\end{subequations}
Now proceed similar to the above.
The case when \cref{prop:opp:inclusion:iii} holds is treated similarly to  
the case when \cref{prop:opp:inclusion:ii} holds.
The proof is complete.
\end{proof}	

Let $C\colon X\rras X$ be monotone. 
 Before we proceed we recall
(see, e.g.,  \cite[Proposition~25.19(i)]{BC2017})
 that
\begin{equation}
	\label{e:3*:dual}
	\text{$C$ is $3^*$ monotone $\siff $
		$C^{-1}$ is $3^*$ monotone
		$\siff $
		$C^{-\ovee}$ is $3^*$ monotone,}	
\end{equation}	
where $C^{-\ovee}=(-\Id)\circ C^{-1}\circ (-\Id)$.

We also recall that 
$T$ is self-dual  (see, e.g., \cite[Lemma~3.6~on~page~133]{EckThesis}), i.e.,
%or \cite[Corollary~4.3]{JAT2012})
  \begin{equation}
  	\label{eq:sd}
  	T_{(A,B)}=T_{(A^{-1},B^{-\ovee})}.
  \end{equation}	

\begin{theorem}[the range of $\Id -T$]
\label{main:thm}
The following implications  hold:
\begin{enumerate}
	\item
	\label{main:thm:i}
	$A$ and $B$ are $3^*$ monotone $\RA$ $	\cran(\Id-T)=\overline{D\cap R}$.
	\item
	\label{main:thm:ii}
	$(\exists C\in \{A,B\})$ $\dom C=X$  and  $C$ is $3^*$ monotone
	$\RA$ $	\cran(\Id-T)=\overline{R}$.
	\item
	\label{main:thm:iii}
	$(\exists C\in \{A,B\})$ $\ran C=X$  and  $C$ is $3^*$ monotone
	$\RA$ $	\cran(\Id-T)=\overline{D}$.
\end{enumerate}	
%Then 
%\begin{equation}
%	\cran(\Id-T)=\overline{D\cap R}.
%\end{equation}	
\end{theorem}	
\begin{proof}
	\cref{main:thm:i}:
	It follows from \cref{prop:opp:inclusion}\cref{prop:opp:inclusion:i}
	that 	$		\overline{D\cap R}\subseteq \cran(\Id-T).$
	Now combine this with  \cref{eq:inc:always}.
	\cref{main:thm:ii}:
%	Suppose that \cref{main:thm:i} (respectively \cref{main:thm:ii}) holds.
Observe that in this case $D=\dom A-\dom B=X$, hence $D\cap R=R$.
First suppose that $C=B$.
	It follows from \cref{prop:opp:inclusion}\cref{prop:opp:inclusion:ii}
	that $\overline{R}=\overline{D\cap R}\subseteq \cran(\Id-T)$.
	Now combine with  \cref{eq:inc:always}. 
To prove the claim in the case $C=A$ proceed as above but use 	
 \cref{prop:opp:inclusion}\cref{prop:opp:inclusion:iii}.
\cref{main:thm:iii}:	
Observe that 
\cref{e:3*:dual} implies 
[$(\exists C\in \{A,B\})$ $\ran C=X$  and  $C$ is $3^*$ monotone]
$\siff $ [$(\exists \widetilde{C}\in \{A^{-1},B^{-\ovee}\})$ $\dom  \widetilde{C}=X$  
and  $ \widetilde{C}$ is $3^*$ monotone]
by, e.g., \cite[Proposition~25.19(i)]{BC2017}.
 Now combine with \cref{main:thm:ii} applied with $(A,B)$
  replaced by $(A^{-1},B^{-\ovee})$
  in view of \cref{eq:sd}.
\end{proof}	
\begin{theorem}[the range of $T$]
		\label{main:T:thmm}
The following implications  hold:
\begin{enumerate}
	\item
	\label{main:T:thmm:i}
	$A$ and $B$ are $3^*$ monotone
	$\RA$
	$\cran T=\overline{(\dom A-\ran B)\cap(\ran A+\dom B)}$.
	\item
	\label{main:T:thmm:ii}
	$(\exists C\in \{A,B^{-1}\})$ $\dom C=X$  and  $C$ is $3^*$ monotone
	$\RA$ $	\cran T=\overline{\ran A+\dom B}$.
	\item
	\label{main:T:thmm:iii}
	$(\exists C\in \{A,B^{-1}\})$ $\ran C=X$  and  $C$ is $3^*$ monotone
	$\RA$ $	\cran T=\overline{\dom A-\ran B}$.
\end{enumerate}		
%	Then 
%	\begin{equation}
%	\cran T=\overline{(\dom A-\ran B)\cap(\ran A+\dom B)}
%	\end{equation}
\end{theorem}	
\begin{proof}
It is straightforward to verify that 
$T=\Id-(\Id-\J{A}+\J{B^{-1}}\R{A})=\Id-T_{(A,B^{-1})}$.
 Now apply \cref{main:thm}\cref{main:thm:i}--\cref{main:thm:iii} with $B$
 replaced by $B^{-1}$ in view of \cref{e:3*:dual} 
 applied with  $C$
 replaced by $B$.	 
\end{proof}	

The assumptions in \cref{main:thm}
are critical as we illustrate below.
\begin{example}
\label{ex:cne:1}
Suppose that $S\colon X\to X$ is continuous, linear, and single-valued such
that  $S$ and $-S$ are monotone and $S^2=-\gamma \Id$
where $\gamma > 0$.
Set $(A,B)=(S,-S)$.
Then the following hold:
\begin{enumerate}
\item
\label{ex:cne:1:i}
$(\forall x\in X)$ $\scal{x}{Sx}=0$.
\item
\label{ex:cne:1:ii}
Neither $S$ nor $S^*$ is $3^*$	monotone.
\item
\label{ex:cne:1:ii:iii}
$\J{A}	=\tfrac{1}{1+\gamma }( \Id-  S)$
and 
$\J{B}	=\tfrac{1}{1+\gamma }( \Id+  S)$.
\item
\label{ex:cne:1:iii}
$\R{A}=\tfrac{1}{1+\gamma}((1-\gamma)\Id-2S)$	
and 
$\R{B}=\tfrac{1}{1+\gamma}((1-\gamma)\Id+2S)$.
\item
\label{ex:cne:1:iv}
$\R{B}\R{A}=\Id$. Hence, $T=\Id$ and consequently $\Id-T\equiv 0$.
\item
\label{ex:cne:1:v}
$\{0\}=\ran(\Id-T)=\cran(\Id-T)\subsetneqq X=\overline{D\cap R}$.	
\end{enumerate}	
\end{example}
\begin{proof}
\cref{ex:cne:1:i}\&\cref{ex:cne:1:ii}:	
This is clear.
\cref{ex:cne:1:ii:iii}:
Indeed, observe that 
$\J{A}\J{B}
=(\Id+S)^{-1}(\Id-S)^{-1}
=((\Id-S)(\Id+S))^{-1}
=((1+\gamma)\Id)^{-1}
=(1+\gamma)^{-1}\Id$.
Therefore, $\J{A}=(1+\gamma)^{-1}(\Id-S)$
and  $\J{B}=(1+\gamma)^{-1}(\Id+S)$
as claimed.
\cref{ex:cne:1:iii}:
This is a direct consequence of \cref{ex:cne:1:ii:iii}.
\cref{ex:cne:1:iv}:	
Using \cref{ex:cne:1:iii}
we have 
$\R{B}\R{A}
=(1+\gamma)^{-2}((1-\gamma)\Id-2S)((1-\gamma)\Id+2S)
=(1+\gamma)^{-2}((1-\gamma)^2\Id-4S^2)
=(1+\gamma)^{-2}((1-\gamma)^2\Id+4\gamma \Id)=\Id$.			
\cref{ex:cne:1:v}:	
Clearly, $\dom A=\dom B=X$.
Moreover, by assumption $A^{-1}=-B^{-1}=-\gamma^{-1} S$.
Hence, $A$ and $B $ are surjective and 
we conclude that $\ran A=\ran B=X$
 and the conclusion follows. 
\end{proof}
	
\begin{example}
	\label{ex:cne:2}
	Suppose that $X=\RR^2 $.
	Let $u\in \RR^2$
	be such that $\norm{u}=1$.
	Let $A\colon \RR^2\to \RR^2\colon (\xi_1,\xi_2)\mapsto (-\xi_2,\xi_1)$
	(the rotator in the plane  by $\pi/2$)
	and set\footnote{Let $C$ be a nonempty closed convex subset of $X$.
	Here and elsewhere we shall use $\iota_C$
(respectively $N_C$) to denote the 
\emph{indicator function} (respectively the \emph{normal cone operator})
associated with $C$.} $B=N_{\RR\cdot u}$.
	Then 
	$A$ is \emph{not} $3^*$ monotone,
	$\J{A}=\tfrac{1}{2}(\Id-A)$,
	$\R{A} =-A$	
	and $\J{B}=\Pj{\RR\cdot u}=\scal{\cdot}{u}u$.
	Moreover, 
	the following hold:
	\begin{enumerate}
		\item
		\label{ex:cne:2:i}
		$\dom A=\ran A=\RR^2$, $\dom B=\RR\cdot u$,
		$\ran B=\{ u\}^\sperp$.
		\item
		\label{ex:cne:2:ii}
		$\overline{D\cap R}=D\cap R=\RR^2$.
		\item
		\label{ex:cne:2:iii}
		$ \Id-T=
		\tfrac{1}{2}\Pj{\RR\cdot \J{A} u}
		$.
		\item
		\label{ex:cne:2:iii:iv}
		$\ran (\Id-T)=\cran(\Id-T)=
		\RR\cdot \J{A}u
		$.
		\item
		\label{ex:cne:2:iv}
		$	\RR\cdot \J{A}u=\cran(\Id-T)\subsetneqq \overline{D\cap R}=\RR^2$.
	\end{enumerate}	
\end{example}
\begin{proof}
	It is clear that $A$ is \emph{not} $3^*$ monotone,
	that $B=\partial \iota_{\RR\cdot u}$ is  $3^*$ monotone,
that $\dom A=\ran A=\RR^2$
	and that $\dom B=\RR\cdot u$.
	The formulae for $\J{A} $ and  $\R{A}$
	follow from applying \cref{ex:cne:1}\cref{ex:cne:1:ii:iii}\&\cref{ex:cne:1:iii}
    with $\gamma=1$. The formula for 
$\J{B}$ follows from, e.g., \cite[Example~23.4]{BC2017}.
	
	\cref{ex:cne:2:i}\&\cref{ex:cne:2:ii}:
	This is clear.
	
	\cref{ex:cne:2:iii}:
	Set $u=(\alpha,\beta)$
	and observe that $\alpha^2+\beta^2=1$.
	It is straightforward to verify that
	\begin{equation}
		\J{B}=\Pj{\RR\cdot u}=\scal{\cdot}{u}u
		=\begin{pmatrix}
			\alpha^2&\alpha \beta
			\\
			\alpha \beta&\beta^2
		\end{pmatrix}	
		\text{~~hence~~}	
		R_B
		=\begin{pmatrix}
			2	\alpha^2-1&2\alpha \beta
			\\
			2	\alpha \beta&2\beta^2-1
		\end{pmatrix}	.
	\end{equation}	
On the one hand we have 
	\begin{equation}
		\label{eq:form:Id-T}
		\Id-T=\tfrac{1}{2}(\Id-R_B\R{A})
		=\begin{pmatrix}
			\tfrac{1}{2}+\alpha \beta& \tfrac{1}{2}-\alpha^2
			\\
			-\tfrac{1}{2}+\beta^2&\tfrac{1}{2}-\alpha \beta
		\end{pmatrix}
		=	\frac{1}{2}\begin{pmatrix}
			(\alpha +\beta)^2& -(\alpha^2-\beta^2)
			\\
			-(\alpha^2-\beta^2)&(\alpha -\beta)^2
		\end{pmatrix}.
	\end{equation}	
On the other hand, observe that $\J{A}u=\tfrac{1}{2}(u-Au)=\tfrac{1}{2}(\alpha+\beta,\alpha-\beta)$,
 hence $\norm{\J{A}u}^2=\tfrac{1}{2}$. Consequently, 
	\begin{equation}
\Pj{\RR\cdot \J{A}u}=\frac{1}{\norm{\J{A}u}^2}\scal{\cdot}{\J{A}u}\J{A}u
	=\begin{pmatrix}
		\alpha^2&\alpha \beta
		\\
		\alpha \beta&\beta^2
	\end{pmatrix}	
	=\begin{pmatrix}
		(\alpha +\beta)^2& -(\alpha^2-\beta^2)
		\\
		-(\alpha^2-\beta^2)&(\alpha -\beta)^2
	\end{pmatrix},
\end{equation}	
and the conclusion follows.

	\cref{ex:cne:2:iii:iv}:
	This is a direct consequence of \cref{ex:cne:2:iii}.
	\cref{ex:cne:2:iv}:
	Combine \cref{ex:cne:2:ii} and \cref{ex:cne:2:iii:iv}.
\end{proof}

We conclude this section by the following example 
which  shows in a strong way that the closures of the sets 
in \cref{main:thm} are critical.

Let	$U$ and $V$
be nonempty closed convex subset of $X$
such that $U$ is bounded.
We recall that (see, e.g., \cite[Proposition~3.42]{BC2017})
\begin{equation}
	\label{eq:sum:sets:cl}	
	\text{$U+V$ is closed.}	
\end{equation}	
\begin{example}
\label{ex:inf:dim:neq}
Suppose that $\dim X\ge 2$. 
Let $C=\ball{0}{1}$ be the closed unit ball  
 and let $U$ be a closed linear subspace of $X$.
 Set $(A,B)=(N_C,N_U)$. Then the following hold:
 \begin{enumerate}
 \item
 \label{ex:inf:dim:neq:i}
$(\dom A,\dom B,\ran A,\ran B)=(C,U,X,U^\perp)$.
\item
\label{ex:inf:dim:neq:ii}
$\overline{\dom A-\dom B}=\overline{C+U}={C+U}=\dom A-\dom B$.
\item
\label{ex:inf:dim:neq:iii}
$\cran (\Id-T)={C+U}$.
\item
\label{ex:inf:dim:neq:iv}
Set $S=\menge{u+u^\sperp}{u\in U\smallsetminus\{0\}, u^\sperp\in U^\perp, \norm{u^\sperp}=1}$.
Then $S\subseteq (C+U)\smallsetminus \ran (\Id-T)$.
Consequently, $\ran(\Id-T)\subsetneq C+U=\dom A-\dom B$.	
 \end{enumerate}
\end{example}	  
\begin{proof}
\cref{ex:inf:dim:neq:i}:
This is clear in view of, e.g., \cite[Corollary~21.25]{BC2017}.
\cref{ex:inf:dim:neq:ii}:
The first and the third identities follow from 
\cref{ex:inf:dim:neq:i}.
The second identity follows from
\cref{eq:sum:sets:cl}	applied with $V$
replaced by $C$.
\cref{ex:inf:dim:neq:iii}:
Combine 
\cref{ex:inf:dim:neq:i},
\cref{ex:inf:dim:neq:ii}
and \cref{main:thm}\cref{main:thm:iii}.
\cref{ex:inf:dim:neq:iv}:	
Clearly, $S\subseteq C+U$.
Now let $s\in S$.
Then $s=u+u^\sperp, u\in U\smallsetminus\{0\}, u^\sperp\in U^\perp, \norm{u^\sperp}=1$.
Suppose for eventual contradiction that 
$s\in \ran(\Id-T)$. Let $y\in X$ be such that
$s=y-Ty$.
Because $\Pj{U}$ is linear by, e.g., 
\cite[5.13(1)~on~page~79]{Deutsch}, we have in view of \cref{e:Id-T:desc} 
\begin{equation}
	\label{eq:cn:loc:y}
u+u^\sperp=	\Pj{C} y-2\Pj{U}\Pj{C}y+\Pj{U}y.
\end{equation}	
We proceed by examining the following cases.
\\
\textsc{Case~1:}
$y\in C$.
Then \cref{eq:cn:loc:y} yields 
$u+u^\sperp=	y-2\Pj{U}y+\Pj{U}y=\Pj{U^\perp }y\in U^\perp$.
That is, $u\in (-u^\sperp+U^\perp)\cap U=U^\perp\cap U$, hence $u=0$
which is absurd.
\\
\textsc{Case~2:}
$y\not\in C$.
In this case $\Pj{C}y=y/\norm{y}$ by, e.g., \cite[Example~3.18]{BC2017}.
Therefore,
\cref{eq:cn:loc:y} yields 
\begin{equation}
	\label{eq:cn:loc:y:i}
	u+u^\sperp=	\frac{ 1}{\norm{y}}y-2\frac{ 1}{\norm{y}}\Pj{U}y+\Pj{U}y
	=\frac{ 1}{\norm{y}}\Pj{U^\perp}y+\Big(1-\frac{ 1}{\norm{y}}\Big)\Pj{U}y.
\end{equation}	
That is, $u^\sperp=\Pj{U^\perp}y/\norm{y}$.
On the one hand, the above argument implies that  
$1=\norm{u^\sperp}=\norm{\Pj{U^\perp}y}/\norm{y}$.
Therefore,
$\norm{\Pj{U^\perp}y}^2=\norm{y}^2=\norm{\Pj{U}y}^2+\norm{\Pj{U^\perp}y}^2$.
Hence, $\Pj{U}y=0$.
On the other hand,
\cref{eq:cn:loc:y:i} implies that  
$u=(1-(1/\norm{y}))\Pj{U}y$. Altogether, we conclude that 
$u=0$ which is absurd.
Therefore, $s\not\in \ran(\Id-T) $. The proof is complete.
\end{proof}	

\section{The case $(A,B)=(\partial f,\partial g)$}
\label{sec:3}
In the remainder of this paper we assume that 
\begin{empheq}[box=\mybluebox]{equation}
	\label{e:def:fg}
	\text{$f$
		and $g$
		 are proper  lower semicontinuous convex functions on $X$ .}
\end{empheq}
We use the abbreviations
\begin{empheq}[box=\mybluebox]{equation}
	\label{e:def:fg:abb}
	\big(\prox_f,\prox_{f^*}, \prox_g,\R{f}\big)
	=
\big(\operatorname{Prox}_f,\operatorname{Prox}_{f^*},
\operatorname{Prox}_g,2\operatorname{Prox}_f-\Id\big).
\end{empheq}
In this case 
\begin{equation}
\T_{(\partial f, \partial g)}=\Id-\J{\partial f}
+\J{\partial g}\R{\partial f}=\Id-\prox_f+\prox_g\R{f}.	
\end{equation}	
The following simple lemma is stated in 
\cite[page~167]{Br-H}. 
We state the proof for the sake of completeness.
\begin{lemma}
	\label{lem:3*:fg}
Let $(y,z^*)\in \dom f\times \dom f^*$.
Then
\begin{equation}
\label{eq:conc:fg:*}
\inf_{(x,x^*)\in \gra \partial f }\scal{x-y}{x^*-z^*}>-\infty.	
\end{equation}	
Consequently, $\partial f$ is $3^* $ monotone.	
\end{lemma}	
\begin{proof}
Let $(x,x^*)\in \gra \partial f$. 
It follows from the subgradient inequality that
$f(y)\ge f(x)+\scal{x^*}{y-x}=f(x)+\scal{x^*-z^*}{y-x}+\scal{z^*}{y-x}$.
Rearranging yields
\begin{subequations}
\begin{align}
\scal{x-y}{x^*-z^*}
&\ge f(x)-f(y)-\scal{z^*}{x}+\scal{y}{z^*} =-f(y)-(\scal{z^*}{x}-f(x))+\scal{y}{z^*} 
\\ 
&\ge -f(y)-f^*(z^*)+\scal{y}{z^*}.
\end{align}		
\end{subequations}	
This verifies 	\cref{eq:conc:fg:*}.
The $3^*$ monotonicity of $\partial f$
follows from combining \cref{eq:conc:fg:*}  and  the fact that
$\dom \partial f \subseteq \dom f$
and 
$\dom \partial f^* \subseteq \dom f^*$.
\end{proof}	

\begin{proposition}
	\label{prop:opp:inclusion:fg}
	Let $w\in ({\dom} f-{\dom} g)\cap ({\dom} f^*+{\dom} g^*)$.
    Set $(A,B)=(\partial f,\partial g)$.
	Then $w\in \overline{\ran} (\Id-T)$.
\end{proposition}	

\begin{proof}
Proceeding similar to the proof  \cref{prop:opp:inclusion}, 
let 
$(x_n)_{n\ge 1}$ be such that 
$(x_n, w-\tfrac{1}{n^2}x_n)_{n\ge 1}$ lies in $\gra (\Id-T)$.
Our goal is to show that 
\begin{equation}
	\label{w:eq}
\tfrac{1}{n^2}x_n\to 0.
\end{equation}	
Obtain
$(a,b,a^*,b^*)$
from $ \dom f\times \dom g\times\dom f^*\times \dom g^*$
such that 
\begin{equation}
	\label{e:w:connect:fg}
	w=a-b=a^*+b^*.
\end{equation}	
Recall that \cref{lem:gr:Id-T} and \cref{e:w:connect:fg} imply
\begin{align}
	\label{e:sq:version:fg}
	\scal{x-(a+a^*)}{x-Tx -w}
	&\ge 
	\scal{\prox_fx-a}{\prox_{f^*}x-a^*} 
	+\scal{\prox_g\R{f}x-b}{\prox_{g^*}\R{f}x-b^*}.
\end{align}
It follows from \cref{lem:3*:fg} applied to $f$ and $g$
respectively 
that 
\begin{equation}
	\label{eq:prox:bd}
	\inf_{x\in X}\scal{\prox_fx-a}{\prox_{f^*}x-a^*} >-\infty 
	\text{~~and~~} \inf_{x\in X}\scal{\prox_g\R{A} x-b}{\prox_{g^*}\R{A}x-b^*}>-\infty.
\end{equation}	
Combining \cref{eq:prox:bd}  with \cref{e:sq:version:fg} 
in view of \cref{e:w:connect:fg} we learn that
\begin{subequations}
	\label{e:sq:version:inf:fg}
	\begin{align}
		&\qquad \inf_{x\in X}\scal{x-(a+a^*)}{x-Tx -(a-b)}
		\nonumber
		\\
		&\ge 
		\inf_{x\in X}\big(\scal{\prox_fx-a}{\prox_{f^*}x-a^*}  
		+\scal{\prox_g\R{f} x-b}{\prox_{g^*}\R{f}x-b^*}\big)
		\\
		&\ge 
		\inf_{x\in X}\scal{\prox_fx-a}{\prox_{f^*}x-a^*} 
		 +\inf_{x\in X}\scal{\prox_g\R{f} x-b}{\prox_{g^*}\R{f}x-b^*}>-\infty.
	\end{align}
\end{subequations}
Now proceed similar to the proof of 
\cref{prop:opp:inclusion}\cref{prop:opp:inclusion:i}
to conclude that $(x_n/n)_{n\ge 1}$ is bounded hence 
\cref{w:eq} holds.
\end{proof}	

\begin{theorem}
	\label{thm:main:fg}
    Set $T=T_{(\partial f,\partial g)}$. Then the following hold:
    \begin{enumerate}
    	\item
    	\label{thm:main:fg:i}
    	$\cran (\Id-T)
    	=\overline{(\dom  f-\dom g)
    		\cap(\dom  f^*+\dom  g^*)}$
    	\item
    	\label{thm:main:fg:ii}
    	$\cran T=\overline{(\dom  f-\dom g^*)
    		\cap(\dom f^*+\dom  g)}$
    \end{enumerate}	
\end{theorem}	
\begin{proof}
\cref{thm:main:fg:i}:
Indeed, by \cref{prop:opp:inclusion:fg}
 and \cref{eq:inc:always} applied with 
$(A,B)$ replaced with $(\partial f, \partial g)$
we have 
\begin{subequations}
	\label{eq:dom:fg:subdif}
\begin{align}
&\quad\overline{(\dom  f-\dom g)
\cap(\dom  f^*+\dom  g^*)}
\\
&\subseteq \cran(\Id-T)	
\\
&\subseteq\overline{(\dom  \partial f-\dom\partial g)
	\cap(\dom\partial  f^*+\dom\partial  g^*)}
\\
&\subseteq \overline{(\dom  f-\dom g)
	\cap(\dom  f^*+\dom  g^*)}.
\end{align}		
\end{subequations}	
\cref{thm:main:fg:ii}:
Observe that $T=\Id-T_{(\partial f,\partial g^*)}$.
Now combine with \cref{thm:main:fg:i}
applied with $g$ replaced by $g^{*}$.
This completes the proof.
 \end{proof}	

As a byproduct of the above results we
obtain the following corollary.

\begin{corollary}
We have
\begin{equation}
\overline{(\dom  \partial f-\dom\partial g)
	\cap(\dom\partial  f^*+\dom\partial  g^*)}
=	
\overline{(\dom  f-\dom g)
	\cap(\dom  f^*+\dom  g^*)}.
\end{equation}		
\end{corollary}
\begin{proof}
This is s direct consequence of \cref{eq:dom:fg:subdif}.	
\end{proof}

\begin{example}
\label{ex:twosets:onebdd}
Suppose that  $U$ and $V$
are  nonempty closed convex subset of $X$.
Set $f=\iota_U$
 and $g=\iota_V$.
 Suppose that
 $U$ is bounded.
 Then the following hold:
\begin{enumerate}
\item
\label{ex:twosets:onebdd:i}
$\cran(\Id-T)=U-V$.
\item
\label{ex:twosets:onebdd:ii}
$\cran T=U-(\rec V)^\ominus$.
\end{enumerate}	
\end{example}	
\begin{proof}
Observe that $(\dom f, \dom g)
=(\overline{\dom} f, \overline{\dom} g)
=(U,V)$.
Moreover, $(\dom f^*,\overline{\dom} g^* )=(X, (\rec V)^\ominus)$.
\cref{ex:twosets:onebdd:i}:
It follows from \cref{thm:main:fg}\cref{thm:main:fg:i}
that  
$\cran(\Id-T)=\overline{(U-V)\cap X}
=\overline{U-V}$.
Now combine with \cref{eq:sum:sets:cl}
applied with $V$ replaced by $-V$.
\cref{ex:twosets:onebdd:ii}:	
It follows from \cref{thm:main:fg}\cref{thm:main:fg:ii}
that  
$\cran T=\overline{\dom f -\dom g^*}
=\overline{\overline{\dom} f -\overline{\dom} g^*}
=\overline{U-(\rec V)^\ominus}
$.	
Now combine with \cref{eq:sum:sets:cl}
applied with $V$ replaced by $-(\rec V)^\ominus$.
\end{proof}

\section{The sets $D$ and $R$ and 
	their corresponding minimal norm vectors} 
\label{sec:4}
Because $A$, $B$ and $\Id-T$ are maximally monotone,
we know that the sets $\dom A$, $\dom B$, 
$\ran A$,  $\ran B$ and $\ran (\Id-T)$ have \emph{convex closures}
(see  \cite[Theorem~31.2]{Simons2})
 i.e., 
\begin{equation}
	\label{eq:con:cl}
	\text{$\overline{D}$, $\overline{R}$ and $\cran(\Id-T)$ are convex. }	
\end{equation}	
Consequently, the following vectors 
\begin{equation}
	\label{e:def:vD:vR}
	v= P_{\cran  (\Id-T)}(0),
	\quad
	v_D= P_{\overline{\dom A-\dom B}}(0),
	\quad
	v_R= P_{\overline{\ran A+\ran B}}
	(0)
\end{equation}	
are well defined.
\begin{remark}\
\label{rem:intersection}
\begin{enumerate}
\item
\label{rem:intersection:i}
If $X$ is finite-dimensional, then we know more (see, e.g, 
\cite[Lemma~5.1(i)]{MOR}):
$D$ and $R$ are nearly convex; 
so in particular
by \cite[Theorem~2.16]{BMW}, $\overline{D} = \overline{\ri D}$
and $\overline{R} = \overline{\ri R}$.
Moreover,  Minty's theorem applied to $B$ 
yields $\ri D-\ri R=\ri (D-R)=\ri (\dom A-\dom B-\ran A-\ran B)=\ri X=X$.
Hence, $\ri D \cap \ri R\neq \fady$.  Consequently, we learn that 
\begin{equation}
	\label{eq:sets:cl:eq}
\overline{D\cap R}=\overline{D}\cap \overline{R}.
\end{equation}	
\item
\label{rem:intersection:ii}
We do not know whether or not such an identity \cref{eq:sets:cl:eq}
survives in infinite-dimensional Hilbert spaces.
Indeed, in view of  \cref{prop:affine:v} below,
on the one hand, 
any counterexample must feature that neither of the operators is 
an affine
relation or that
none of the operators has a bounded domain or a bounded range.
On the other hand, \cref{prop:fat:sets} implies that
one has to avoid scenarios when 
both $D$ and $R$ has an nonempty interior.

\end{enumerate}
\end{remark}

The next result provides some sufficient conditions where
\cref{eq:sets:cl:eq} holds  in infinite-dimensional Hilbert spaces.
\begin{proposition}
	\label{prop:affine:v}
	Suppose that one of the following conditions hold:
	\begin{enumerate}
		\item
		\label{prop:affine:v:i}
		$(\exists C\in \{A,B\})$ such that $\dom C$ and $\ran C$ are affine.
		\item
		\label{prop:affine:v:i:i}
		$(A,B)=(N_K,N_L)$ where 
		$K$ and $L$ are nonempty closed convex cones of $X$
		and $K^\ominus+L^\ominus $ is closed.
		\item
		\label{prop:affine:v:ii}
		$(\exists C\in \{A,B\})$ $\dom C$ is bounded  or $\ran C$ is bounded.
		\item
		\label{prop:affine:v:iii}
		$(\exists C\in \{A,B\})$  %$C$ is $3^*$ monotone and 
		$\dom C=X$ or $\ran C=X$.
	\end{enumerate}
	Then $\overline{D\cap R}=\overline{D}\cap \overline{R}$.
\end{proposition}	

\begin{proof}
	\cref{prop:affine:v:i}:
	Indeed, suppose that $C=A$.  
	Let $w\in \overline{D}\cap \overline{R}$ and 
	let $(a_n,b_n,a_n^*,b_n^*)_\nnn$ be a sequence in 
	$\dom A\times\dom B\times \ran A\times \ran B$
	such that $a_n-b_n\to w$ and $a_n^*+b_n^*\to w$.
	Set $(\forall \nnn)$ $w_n= \J{A} (a_n^*+b_n^*)+\J{A^{-1}}(a_n-b_n)$.
	Because $\J{A}$, as is $\J{A^{-1}}$, is firmly nonexpansive, it is continuous.
	Therefore 
	\begin{equation}
		\label{e:wn:loc:i}
		w_n\to \J{A} w+\J{A^{-1}}w=w.
	\end{equation}
	We claim that 
	\begin{equation}
		\label{e:wn:loc}
		\text{$(w_n)_\nnn$ lies in $D\cap R$}. 	
	\end{equation}	
	Indeed, on the one hand because $\dom A=\ran \J{A}$ is affine we have 
	\begin{subequations}
		\begin{align}	
			w_n
			&= \J{A} (a_n^*+b_n^*)+a_n-b_n-\J{A}(a_n-b_n)
			\\
			&=\J{A}  (a_n^*+b_n^*)+a_n-\J{A} (a_n-b_n)-b_n \in \dom A-\dom B=D.
		\end{align}
	\end{subequations}	
	On the other hand, because $\ran A=\ran \J{A^{-1}}$ is affine we have 
	\begin{subequations}
		\begin{align}	
			w_n
			&= \J{A^{-1}} (a_n-b_n)+a_n^*+b_n^*-\J{A^{-1}} (a_n^*+b_n^*)
			\\
			&=\J{A^{-1}} (a_n-b_n)+a^*_n-\J{A^{-1}}(a_n^*+b_n^*)+b^*_n \in \ran A+\ran B=R.
		\end{align}
	\end{subequations}	
	This proves \cref{e:wn:loc}. Now combine with \cref{e:wn:loc:i}.
	The proof in the case $C=B$ is similar. 
	
	\cref{prop:affine:v:i:i}:
	Let $w\in \overline{D}\cap \overline{R}=\overline{(K-L)}\cap({K^\ominus+L^\ominus})$
	and 
	let $(k_n,l_n)_\nnn$ be a sequence in 
	$K\times L $
	such that $		k_n-l_n\to w$.
Set $(\forall \nnn)$ $w_n= \Pj{(K\cap L)^\ominus}(k_n-l_n)$.
Observe that because 
$(K\cap L)^\ominus=\overline{K^\ominus+L^\ominus }=K^\ominus+L^\ominus $
by, e.g., \cite[remarks~on~page~48]{Deutsch},
we have  
\begin{equation}
	\label{e:lim:cone:w}
	w_n\to \Pj{(K\cap L)^\ominus}w=w.
\end{equation}	
On the one hand, by construction  
$(w_n)_\nnn$ lies in $K^\ominus+L^\ominus $.
On the other hand we have $(\forall \nnn)$
$	 w_n=	\Pj{(K\cap L)^\ominus}(k_n-l_n)=k_n-l_n-\Pj{K\cap L}(k_n-l_n)
=k_n-(l_n+\Pj{K\cap L}(k_n-l_n))\in K-L.$
Hence,
 $(w_n)_\nnn$ lies in $(K-L)\cap(K^\ominus+L^\ominus )$.
Combining this with
\cref{e:lim:cone:w} we learn that 
$w\in  \overline{({K-L})\cap({K^\ominus+L^\ominus})}$.
	\cref{prop:affine:v:ii}:
	If $\dom C$ is bounded
	then \cite[Corollary~21.25]{BC2017} 
	implies that $\ran C=X$, hence $R=X$. 
	Therefore, 
	$\overline{D\cap R}=\overline{D}=\overline{D}\cap X
	=\overline{D}\cap\overline{R}$.
	The case when $\ran C$ is bounded follows 
	similarly by applying the previous argument to $C^{-1}$.	
	\cref{prop:affine:v:iii}:	
	If $\dom C=X$, then $D=X$. 
	Therefore, 
	$\overline{D\cap R}=\overline{R}=X\cap\overline{R}
	=\overline{D}\cap\overline{R}$.
	The case when $\ran C=X$  follows 
	similarly. 
\end{proof}

Before we proceed we recall the following facts.
\begin{fact}[Simons]
	\label{fact:domdom}
	Suppose that $\intr D\neq \fady$.
	Then $\overline{D}=\overline{\intr D}$.
\end{fact}
\begin{proof}
	See \cite[Theorem~22.1(c)~and~Theorem~22.2(a)]{Simons1}.	
\end{proof}	
\begin{fact}
	\label{fact:pazy}
	Let $S$ be a nonempty closed convex subset of $X$,
	let $w=P_{S}0$,
	and let $s\in S$.
	If $\norm{s}\le \norm{w}$ then $s=w$.
\end{fact}
\begin{proof}
	This is \cite[Lemma~1]{Pazy1970}.
\end{proof}	

When the sets 
$D$ and $R$ are reasonably \emph{fat}; namely, 
when $\intr D\neq \fady$
and $\intr R\neq \fady$ we obtain another sufficient condition for 
the conclusion $\overline{D\cap R}=\overline{D}\cap \overline{R}$ as we see in 
\cref{prop:fat:sets} below.

\begin{proposition}
	\label{prop:fat:sets}
	Suppose that $A$ and $B$ are $3^*$ monotone.
	Suppose that		$\intr D\neq \fady $ and $\intr R\neq \fady$.
	Then $\overline{D\cap R}=\overline{D}\cap \overline{R}$.
\end{proposition}
\begin{proof}
	It follows from \cref{fact:domdom} applied to
	$(A, B)$ (respectively $(A^{-1},B^{-\ovee})$)
	that $\overline{D}=\overline{\intr D} $
	(respectively $\overline{R}=\overline{\intr R} $).
	Let $w\in \overline{D}\cap \overline{R}$.
	By \cref{fact:domdom}
	$\overline{D}\cap \overline{R}=\overline{\intr D}\cap \overline{\intr R}$,
	 hence $w\in \overline{\intr D}\cap \overline{\intr R}$	and therefore
	\begin{equation}
		\label{eq:seq:shift}
		\text{there exists a sequence $(d_n,r_n)_\nnn$ in ${\intr D}\times{\intr R}$ such that $(d_n,r_n )\to (w,w)$}.
	\end{equation}	
	Observe that \cref{eq:seq:shift} implies that $(\forall \nnn)$
	$0\in \intr(\dom A-(d_n+\dom B))=\intr(\dom A-\dom B(\cdot-d_n))$. Therefore, by e.g., \cite[Corollary~25.5(iii)]{BC2017}
	we learn that  $(\forall \nnn)$ $A+B(\cdot-d_n)$ is maximally monotone.
	On the one hand, it follows from the $3^*$ monotonicity of
	$A$ and $B$ in view of the celebrated Brezis--Haraux theorem
	(see \cite[Th\'eor\`eme~3]{Br-H} and also \cite[Corollary~31.6]{Simons2})
	that $(\forall \nnn)$
	\begin{subequations}
		\begin{align}
			r_n&\in \intr (\ran A+\ran B)=\intr (\ran A+\ran B(\cdot -d_n))
			\\
			&=\intr \ran (A+B(\cdot -d_n))	.
		\end{align}		
	\end{subequations}	 
	Therefore,  there exist sequences $(x_n,u_n)_\nnn$ in $ \gra A$
	and $(x_n-d_n,-u_n+r_n )_\nnn$ in $ \gra B$.
	Using Minty's theorem \cref{thm:minty}
	we rewrite this as
	\begin{subequations}
		\label{e:seq:gra:AB}
		\begin{align}
			(x_n,u_n)&=(\J{A} (x_n+u_n),\J{A^{-1}}(x_n+u_n))		
			\label{e:seq:gra:AB:i}
			\\
			(x_n-d_n,-u_n+r_n)&=(\J{B}(x_n-u_n-d_n+r_n),\J{B^{-1}}(x_n-u_n-d_n+r_n)).
			\label{e:seq:gra:AB:ii}
		\end{align}		
	\end{subequations}	 
It follows from  \cref{e:seq:gra:AB:i} and \cref{e:seq:gra:AB:ii} that
	\begin{subequations}
		\begin{align}
			d_n&=\J{A} (x_n+u_n)-\J{B}(x_n-u_n-d_n+r_n)	
			\label{e:seq:gra:AB:iii}
			\\
			r_n&=\J{A^{-1}}(x_n+u_n)+\J{B^{-1}}(x_n-u_n-d_n+r_n)	.
			\label{e:seq:gra:AB:iv}
		\end{align}		
	\end{subequations}	 
	Now, set $(\forall \nnn)$
	\begin{equation}
		\label{eq:loc:zn}
		z_n= \J{A} (x_n+u_n)-\J{B}\R{A}(x_n+u_n)\in \ran (\Id-T).
	\end{equation}	
	We claim that $(\forall \nnn)$
	\begin{equation}
		\label{e:fne:B}
		\norm{z_n-r_n}^2+ \norm{z_n-d_n}^2\le 	\norm{d_n-r_n}^2.
	\end{equation}	
	Indeed, using \cref{eq:loc:zn}, \cref{e:seq:gra:AB:iv},
	the firm nonexpansiveness of $J_{B^{-1}}$,
	\cref{e:seq:gra:AB:i}, \cref{eq:loc:zn}
	and \cref{e:seq:gra:AB:iii} we obtain  
	\begin{subequations}
		\begin{align}
			\norm{z_n-r_n}^2
			&=\norm{\J{A} (x_n+u_n)-\J{B}\R{A}(x_n+u_n)-\J{A^{-1}}(x_n+u_n)-\J{B^{-1}}(x_n-u_n-d_n+r_n)}^2
			\\
			&=\norm{\J{B^{-1}}\R{A}(x_n+u_n)-\J{B^{-1}}(x_n-u_n-d_n+r_n)}^2
			\\
			&\le \norm{\R{A}(x_n+u_n)-(x_n-u_n-d_n+r_n)}^2
			\nonumber
			\\
			&\quad-\norm{\J{B}\R{A}(x_n+u_n)-\J{B}(x_n-u_n-d_n+r_n)}^2
			\\
			&= \norm{(x_n-u_n)-(x_n-u_n-d_n+r_n)}^2- \norm{(x_n-z_n)-(x_n-d_n)}^2
			\\
			&= \norm{d_n-r_n}^2- \norm{d_n-z_n}^2.
		\end{align}		
%		\label{e:zn:bdd}
	\end{subequations}
	This proves \cref{e:fne:B}. Taking the limit as $n\to \infty$
	in \cref{e:fne:B}
	in view of \cref{eq:seq:shift} we learn that $z_n\to w$.
	Therefore, in view of \cref{eq:loc:zn}, we learn that $w\in \cran(\Id-T)$.
	Hence, $\overline{D}\cap \overline{R}\subseteq  \cran(\Id-T)$.
	Now combine this with \cref{eq:inc:always}
	and recall that $\overline{D\cap R}\subseteq \overline{D}\cap \overline{R}$.
	The proof is complete.
\end{proof}

We now turn to the minimal norm vectors in the sets 
$\overline{D}$, $\overline{R}$ and $\cran(\Id-T)$;
namely,
$v_D$, $v_R$ and $v$ respectively.
Before we proceed we recall the following useful fact.

\begin{fact}
	\label{f:geo2sets}
	Let $U$ and $V$ be nonempty closed convex subsets of $X$.
	Then 
	\begin{equation}
		\Pj{\overline{U-V}}(0)\in\overline{(\Pj{U}-\Id)(V)}\cap\overline{(\Id-\Pj{V})(U)} 
		\subseteq 
		(-\rec U)^\ominus \cap(\rec V)^\ominus. 
	\end{equation}
\end{fact}
\begin{proof}
	This follows from 
	\cite[Corollary~4.6]{BB94} and \cite[Theorem~3.1]{Zara}. 
\end{proof}

\begin{fact}
	\label{lem:rec:dom:ran}
	Let $A\colon X\rras X$ be maximally monotone.
	Then the following hold:
	\begin{enumerate}
		\item
		\label{lem:rec:dom:ran:i}
		$(\rec \cdom \ A)^\ominus\subseteq \rec (\cran   A)$.
		\item
		\label{lem:rec:dom:ran:ii}
		$(\rec {\cran}\ A)^\ominus\subseteq \rec (\cdom \ A)$.
	\end{enumerate}
\end{fact}	

\begin{proof}
	See \cite[Lemma~3.2]{BM21}
\end{proof}

\begin{lemma}
\label{lem:v:vDvR:abstract}	
Let $S_1$ and $S_2$ be nonempty closed convex subsets of 
$X$ and set $(\forall i\in \{1,2\})$ $v_i=P_{S_i}(0)$.
 Suppose that $\scal{v_1}{v_2}\le 0$ and that $v_1+v_2\in S_1\cap S_2$. 
 Then $v_1+v_2=P_{S_1\cap S_2}(0)$. 
\end{lemma}	
\begin{proof}
Let $s\in S_1\cap S_2$.
In view of the projection theorem see, e.g., 
\cite[Theorem~3.16]{BC2017}
it suffices to show that  $(\forall s\in S_1\cap S_2)$ $\scal{v_1+v_2-0}{v_1+v_2-s}\le 0$.
Indeed, we have 
$
		\scal{v_1+v_2-0}{v_1+v_2-s}
=\scal{v_1}{v_1+v_2-s}+\scal{v_2}{v_1+v_2-s}
\le\scal{v_1}{v_1-s}+\scal{v_2}{v_2-s}
		\le 0+0=0.	
$
\end{proof}

Parts of the following proposition were proved in \cite{BM21}.
We reiterate the proof for the sake of completeness and to avoid
any confusion with the standing assumptions in \cite{BM21}.
\begin{proposition}
	\label{prop:v:decom}
	The following hold:
	\begin{enumerate}
		\item
		\label{prop:v:decom:i}
		$v_D\in (-\rec \cdom  A)^\ominus\cap (\rec \cdom  B)^\ominus
		=(-(\rec \cdom  A)^\ominus)\cap (\rec \cdom  B)^\ominus$.
			\item
		\label{prop:v:decom:iii}
		$v_R\in  (-\rec \cran   A)^\ominus\cap (-\rec \cran   B)^\ominus
		=-((\rec \cran   A)^\ominus\cap (\rec \cran   B)^\ominus)$.	
		\item
		\label{prop:v:decom:ii}
		$v_D\in (-\rec \cran   A)\cap (\rec \cran   B)$.
		\item
		\label{prop:v:decom:iv}
		$v_R\in(-\rec \cdom  A)\cap (-\rec \cdom  B)
		=-(\rec \cdom  A\cap\rec \cdom  B)$.
		\item
		\label{prop:v:decom:v}
		$\scal{v_D}{v_R}=0$.
		\item
		\label{prop:v:decom:v:v}
		$v_D+v_R\in \overline{D}\cap\overline{R}$.
		\item
		\label{prop:v:decom:v:vi}
		$v_D+v_R=\Pj{\overline{D}\cap\overline{R}}(0)$.
	\end{enumerate}		
\end{proposition}	

\begin{proof}
	\cref{prop:v:decom:i}\&\cref{prop:v:decom:iii}:
	Apply \cref{f:geo2sets}
	with $(U,V)$ replaced by
	$(\cdom  A,\cdom  B)$
	(respectively $(\cran   A,-\cran   B)$).
	
	\cref{prop:v:decom:ii}\&\cref{prop:v:decom:iv}:
	Combine \cref{prop:v:decom:i} (respectively \cref{prop:v:decom:iii})
	and \cref{lem:rec:dom:ran}\cref{lem:rec:dom:ran:i}
	(respectively \cref{lem:rec:dom:ran}\cref{lem:rec:dom:ran:ii}).
		
	\cref{prop:v:decom:v}:
	It follows from \cref{prop:v:decom:i} and \cref{prop:v:decom:iv}
	that $(-v_D,-v_R)\in (\rec\cdom  A)^\ominus\times \rec\cdom  A$.
	Hence $\scal{v_D}{v_R}=\scal{-v_D}{-v_R}\le 0$.
	Similarly,  \cref{prop:v:decom:ii} and \cref{prop:v:decom:iii}
	imply that 
	$(v_D,-v_R)\in \rec\cran   B\times(\rec\cran   B)^\ominus$.
	Hence, $-\scal{v_D}{v_R}=\scal{v_D}{-v_R}\le 0$.
	Altogether, $\scal{v_D}{v_R}=0$.
		\cref{prop:v:decom:v:v}:
	Indeed, in view of \cref{prop:v:decom:iv}
	we have $-v_R\in \rec \cdom  B$.
	Therefore,
	${v_D}+{v_R}
	\in \overline{\dom A-\dom B}+{v_R}
	=\overline{\dom A-(-v_R+\cdom B)}
	\subseteq  \overline{\dom A-\cdom  B}
	=\overline{\dom A-\dom B}$.
	Similarly, in view of \cref{prop:v:decom:ii}
	we have $v_D\in \rec \cran   B$.
	Therefore 
	${v_D}+{v_R}
	\in\overline{\ran A+v_D+\cran B}
	\subseteq \overline{\ran A+\cran   B}
	=
	\overline{\ran A+{\ran} B}
	$.
	\cref{prop:v:decom:v:vi}:
	Combine \cref{prop:v:decom:v}, \cref{prop:v:decom:v:v} and \cref{lem:v:vDvR:abstract}
	applied with $(S_1,S_2,v_1,v_2)$ replaced by
	$(\overline{D},\overline{R},v_D,v_R)$.
 \end{proof}	

Before we proceed
 we recall the following
 useful fact.
 \begin{fact}
\label{fact:Pazy:min:prop}
Let $C$ be a nonempty closed convex subset of $X$
and let $w=P_C(0)$.
Suppose that $(u_n)_\nnn$ is a sequence in $C$
such that $\norm{u_n}\to \norm{w}$.
Then $u_n\to w$. 	
 \end{fact}	
 \begin{proof}
 	See  \cite[Lemma~2]{Pazy1970}.
 \end{proof}	  

\begin{lemma}
	\label{lem:nec:suff:v:vDvR}
	Suppose that $\cran(\Id-T)=\overline{D\cap R }$
	(see \cref{main:thm} for sufficient conditions).
Then we  have 
\begin{enumerate}
\item	
	\label{lem:nec:suff:v:vDvR:i}	
$v=v_D+v_R \siff v_D+v_R\in \overline{D\cap R} $.
\item
	\label{lem:nec:suff:v:vDvR:ii}
Suppose $ \overline{D\cap R} = \overline{D} \cap \overline{R} $.
Then $v=v_D+v_R$.
\end{enumerate}	
\end{lemma}	

\begin{proof}
\cref{lem:nec:suff:v:vDvR:i}:
	``$\RA$": This is clear in view of \cref{e:def:vD:vR}.
	``$\LA$":
		Observe that 
\cref{e:assump:free:fd} 
	implies that
	$\norm{v}\le\norm{v_D+v_R}$.
	%	Hence if $v_D+v_R=0$ the $v=0$
	%	and the conclusion follows.
	%	Now suppose that $v_D+v_R\neq 0$
	It follows from \cref{prop:v:decom}\cref{prop:v:decom:v},
	the definition of $v$ and $v_D$
	that
	$\norm{v_D}^2\le \scal{v_D}{\overline{D}}$,
	hence 
	$\norm{v_D}^2\le \scal{v_D}{\overline{D}\cap \overline{R}}$.
	Similarly,
	$\norm{v_R}^2\le \scal{v_R}{\overline{D}\cap \overline{R}}$.	 
	Therefore using 
	Cauchy--Schwarz and \cref{prop:v:decom}\cref{prop:v:decom:v}
	we learn that
	$\norm{v_D+v_R}^2
	=\norm{v_D}^2+\norm{v_R}^2
	\le \scal{v}{v_D}+\scal{v}{v_R}
	=\scal{v}{v_D+v_R}\le \norm{v}\norm{v_D+v_R}$.
	Hence, $\norm{v_D+v_R}\le \norm{v}$.
	Altogether, $\norm{v}=\norm{v_D+v_R}$.
	In view of \cref{e:assump:free:fd}
	and \cref{fact:Pazy:min:prop}, 
	we learn that
	$v=v_D+v_R$.
\cref{lem:nec:suff:v:vDvR:ii}:	
	Combine \cref{lem:nec:suff:v:vDvR:i}
	and \cref{prop:v:decom}\cref{prop:v:decom:v:vi}.
\end{proof}

\begin{proposition}
	\label{prop:sc:nc}
Suppose that $U$ and $V$ are nonempty closed convex subsets of $X$.
Set $(A,B)=(N_U,N_V)$.
Then the following hold:
\begin{enumerate}
\item
\label{prop:sc:nc:i}
$v_R=0$.
\item
\label{prop:sc:nc:ii}
$v_D=v$.
\item	
\label{prop:sc:nc:iii}
$v=v_D+v_R$.	
\end{enumerate}	
	
\end{proposition}		
\begin{proof}
\cref{prop:sc:nc:i}:
If follows from \cite[Theorem~3.1]{Zara} that $(\overline{\ran }A,\overline{\ran} B)=((\rec U)^\ominus,(\rec V)^\ominus)$.
Therefore, $0\in(\rec U)^\ominus+(\rec V)^\ominus\subseteq \overline{R} $. Hence, $v_R=0$.

\cref{prop:sc:nc:ii}:
This is \cite[Proposition~3.5]{Sicon}.

\cref{prop:sc:nc:iii}:
Combine 	\cref{prop:sc:nc:i} and \cref{prop:sc:nc:ii}.
\end{proof}

\begin{corollary}
	\label{prop:cond:for:v}
	Suppose that $\cran(\Id-T)=\overline{D\cap R }$
	(see \cref{main:thm} for sufficient conditions).
	Suppose additionally  that one of the following holds:
	\begin{enumerate}
		\item
		\label{prop:cond:for:v:i}
		$\overline{D\cap R}=\overline{D}\cap \overline{R}$.
		\item
		\label{prop:cond:for:v:ii}
		$X$ is finite-dimensional.
		\item
		\label{prop:cond:for:v:iii}
		$(\exists C\in \{A,B\})$ such that $\dom C$ and $\ran C$ are affine.
		\item
		\label{prop:cond:for:iii:vi}
		$(\exists C\in \{A,B\})$  such that %$C$ is $3^*$ monotone and 
		$\dom C=X$ or $\ran C=X$.
		\item
		\label{prop:cond:for:v:vi}
		$(\exists C\in \{A,B\})$ $\dom C$ is bounded  or $\ran C$ is bounded.
		\item
		\label{prop:cond:for:v:v}
		$A $ and $B$ are $3^*$ monotone,
		$\intr D\neq \fady $ and $\intr R\neq \fady$.
		\item
		\label{prop:cond:for:vii}
		$(A,B)=(N_U,N_V)$, $U$ and $V$ 
		are nonempty closed subsets of $X$.
	\end{enumerate}	
Then $v=v_D+v_R$.
\end{corollary}	

\begin{proof}
	\cref{prop:cond:for:v:i}:
	This is \cref{lem:nec:suff:v:vDvR}\cref{lem:nec:suff:v:vDvR:ii}.  
	\cref{prop:cond:for:v:ii}:
	This follows from combining \cref{rem:intersection}\cref{rem:intersection:i}
	and \cref{prop:cond:for:v:i}.	
	\cref{prop:cond:for:v:iii}--\cref{prop:cond:for:v:vi}:
	Combine \cref{prop:affine:v}\cref{prop:affine:v:i}--\cref{prop:affine:v:iii}
	 and \cref{lem:nec:suff:v:vDvR} \cref{lem:nec:suff:v:vDvR:ii}.	
%	\cref{prop:cond:for:v:vi}:	
%
%	
	\cref{prop:cond:for:v:v}:
	This is \cref{prop:fat:sets}.	
		\cref{prop:cond:for:vii}:
	This is \cref{prop:sc:nc}\cref{prop:sc:nc:iii}
\end{proof}

\begin{remark}
The minimal displacement vector in
	$\cran(\Id-T) $ can be found via (see
	\cite{BBR78}, \cite{Br-Reich77} and \cite{Pazy1970})
	\begin{equation}
		\label{e:limit:v}
	(\forall x\in X)\qquad v=-\lim_{n\to \infty}\frac{T^n x}{n}
	=\lim_{n\to \infty}T^n x-T^{n+1}x.
	\end{equation}	
\end{remark}

Working in $X\times X$
 and recalling \cref{e:def:vD:vR}
we observe that
(see \cite[Proposition~29.4]{BC2017})
\begin{equation}
	\label{e:pds:pr}
	(v_R,v_D)=P_{\overline{R}\times \overline{D}}(0).
\end{equation}

\begin{proposition}[computing $v_D$ and $v_R$]
\label{prop:comp:vdvr}	
Suppose that $v=v_D+v_R$ (see \cref{prop:cond:for:v} for sufficient
conditions). Let $x\in X$.
Then the following hold:
\begin{enumerate}
\item
\label{prop:comp:vdvr:i}	
$v_R=-\lim_{n\to \infty}\frac{\J{A} T^n x}{n}
=\lim_{n\to \infty}(\J{A} T^n x-\J{A} T^{n+1}x)$.
\item
\label{prop:comp:vdvr:ii}	
$v_D=-\lim_{n\to \infty}\frac{\J{A^{-1}}T^n x}{n}
=\lim_{n\to \infty}(\J{A^{-1}}T^n x-\J{A^{-1}}T^{n+1}x)$.	
\end{enumerate}		
\end{proposition}
\begin{proof}
First we verify that
\begin{equation}
\label{eq:ran:inc:sub}
\ran (\J{A} -\J{A} T)\subseteq\overline{R}	
\text{  ~~and~~  }\ran (\J{A^{-1}}-\J{A^{-1}}T)\subseteq\overline{D}.
\end{equation}	
Indeed, observe that 
$\J{A} T+\J{A^{-1}}T=T=\Id-\J{A} +\J{B}\R{A}=\J{A} -\R{A}+\J{B}\R{A}=\J{A} +\J{B^{-1}}\R{A}$.
Hence,
$\J{A} -\J{A} T=\J{A^{-1}}T+\J{B^{-1}}\R{A}$.
Consequently, $\ran (\J{A} -\J{A} T)\subseteq \ran A+\ran B\subseteq\overline{R}$.
Similarly we show that
$\ran (\J{A^{-1}}-\J{A^{-1}}T)\subseteq\overline{D}$.
This verifies \cref{eq:ran:inc:sub}.
Now let $x\in X$, let $ n\ge1$ and observe that \cref{eq:ran:inc:sub} 
and the convexity of $\overline{D}$
and $\overline{R}$
(see \cref{eq:con:cl}) imply
\begin{equation}
	\label{e:ran:split:set}
\{(\J{A}T^n x-\J{A} T^{n+1} x,\J{A^{-1}}T^nx-\J{A^{-1}}T^{n+1}  x ),
\tfrac{1}{n}(\J{A} x-\J{A} T^n x,\J{A^{-1}}x-\J{A^{-1}}T^n x )	
\}\subseteq \overline{R}\times \overline{D}.
\end{equation}	
Now, 
\cref{prop:v:decom}\cref{prop:v:decom:v},
\cref{e:pds:pr},
\cref{eq:ran:inc:sub},
\cref{e:ran:split:set},
the firm nonexpansiveness of $\J{A}$
and \cref{e:limit:v} yield
\begin{subequations}
	\label{se:1}
\begin{align}
\norm{v_R+v_D}^2
&=\norm{v_R}^2+\norm{v_D}^2
=\norm{(v_R,v_D)}^2\le \norm{\tfrac{1}{n}(\J{A} x-\J{A} T^n x,\J{A^{-1}}x-\J{A^{-1}}T^n x )}^2
\\
&= \tfrac{1}{n^2}\norm{\J{A} x-\J{A} T^n x}^2+\tfrac{1}{n^2}\norm{\J{A^{-1}}x-\J{A^{-1}}T^n x }^2
\le \tfrac{1}{n^2}\norm{(x-T^n x)}^2
\\
&=\norm{\tfrac{1}{n}(x-T^n x)}^2\to \norm{v}^2=\norm{v_R+v_D}^2,
\end{align} 
\end{subequations}
 and 
 \begin{subequations}
 		\label{se:2}
 	\begin{align}
 		\norm{v_R+v_D}^2
 		&=\norm{v_R}^2+\norm{v_D}^2
 		=\norm{(v_R,v_D)}^2\le \norm{(\J{A} T^nx-\J{A} T^{n+1} x,\J{A^{-1}}T^nx-\J{A^{-1}}T^{n+1}  x )}^2
 		\\
 		&= \norm{(\J{A} T^nx-\J{A} T^{n+1}  x)}^2+\norm{(\J{A^{-1}}T^nx-\J{A^{-1}}T^{n+1}  x )}^2
 		\\
 		&\le \norm{T^nx-T^{n+1}  x}^2\to \norm{v}^2=\norm{v_R+v_D}^2.
 	\end{align} 
 \end{subequations}
 Therefore we learn  from \cref{se:1} and \cref{se:2} respectively that
  \begin{subequations}
 	\label{se:3}
 	\begin{align}
 	\tfrac{1}{n}\norm{(\J{A} x-\J{A} T^n x,\J{A^{-1}}x-\J{A^{-1}}T^n x )}	&\to \norm{(v_R,v_D)} 	
 	 \label{se:3:i}	
 	\\
 	\norm{(\J{A}T^n x-\J{A} T^{n+1} x,\J{A^{-1}}T^nx-\J{A^{-1}}T^{n+1}  x )}	&\to \norm{(v_R,v_D)}. 	
\label{se:3:ii}	 
\end{align} 
 \end{subequations}
Now  combine
\cref{se:3:i} (respectively \cref{se:3:ii}), 
\cref{eq:ran:inc:sub},
\cref{fact:Pazy:min:prop}
 applied with $X$ replaced by $X\times X$,
$C$ replaced by $\overline{R} \times \overline{D}$
and $w$ replaced by $(v_R,v_D)$
in view of \cref{e:pds:pr} to verify the first (respectively second) identity in 
\cref{prop:comp:vdvr:i} and in \cref{prop:comp:vdvr:ii}.
\end{proof}

We conclude this section with an application 
of our results 
which employ Pierra's technique 
to 
the product space. 

\begin{proposition}
\label{prop:prod:sp}
Suppose that $m\in \stb{2,3,\ldots}$.
For every $i\in\stb{1,2,\ldots,m}$,
suppose that 
$A_i\colon X\rras X$ 
is maximally monotone and $3^*$
monotone.
Set ${\bf\Delta}:=
\stb{(x,\ldots,x)\in X^m~|~x\in X}$,
set ${\bf A}={\ds \times_{i=1}^m} A_i$,
set ${\bf B}=N_{{\bf\Delta}}$,
and set ${\bf T}=T_{({\bf A}, {\bf B})}$.
%let $j\colon X\to X^m\colon x\mapsto(x,x,\ldots,x)$,
%%and 
%%let $e:X^m\to X:(x_1,x_2,\ldots,x_m)
%%\mapsto\tfrac{1}{m}\bk{\sum_{i=1}^m x_i}$.  
%Let ${\bf x}\in X^m$ and suppose that 
%${\bf v}:=P_{{\cran }(\Id-{\bf T})}0\in \ran (\Id-{\bf T})$.
Then the following hold:
\begin{enumerate}
\item
\label{prop:prod:sp:i}
${\bf\Delta}^\perp
	=\menge{(u_1,\ldots,u_m)\in X^m}{\sum_{i=1}^{m}u_i=0}$.
\item
\label{prop:prod:sp:ii}
$
\cran(\Id-{\bf T})=\overline{{ \times_{i=1}^m\dom A_i-\bf\Delta}} \cap \overline{{ \times_{i=1}^m\ran A_i+\bf\Delta}^\perp}.
$
\item
\label{prop:prod:sp:iii}
$
\cran {\bf T}
=\overline{ \times_{i=1}^m\dom A_i-{\bf\Delta}} \cap \overline{\times_{i=1}^m\ran A_i+{\bf\Delta}^\perp}.
$
\end{enumerate}
\end{proposition}

\begin{proof}
It is straightforward to verify that
$\bf A$ is $3^*$ monotone, that
$\dom {\bf A}=\times_{i=1}^m\dom A_i$
 and that 
 $\ran {\bf A}= \times_{i=1}^m\ran A_i$	
\cref{prop:prod:sp:i}: This is
\cite[Proposition~26.4(i)]{BC2017}.
\cref{prop:prod:sp:ii}\&\cref{prop:prod:sp:iii}:
Apply \cref{main:thm}\cref{main:thm:i}
(respectively  \cref{main:T:thmm}\cref{main:T:thmm:i})
with $(A,B)$
replaced by $({\bf A,B})$.
\end{proof}	
\section*{Acknowledgements}
\small
The research of WMM was partially supported by Discovery Grants
of the Natural Sciences and Engineering Research Council of
Canada.

\end{document}